\documentclass[a4paper,12pt]{article}

\usepackage{amsfonts,amssymb,amsmath,amsthm,latexsym,epsfig,amscd,bbm,stmaryrd}

\usepackage[english]{babel}
\makeatletter \@addtoreset{equation}{section}
\makeatother

\makeatletter \@addtoreset{enunciato}{section}
\makeatother

\newcounter{enunciato}[section]

\newtheorem{ittheorem}{Theorem}
\newtheorem{itlemma}{Lemma}
\newtheorem{itproposition}{Proposition}
\newtheorem{itdefinition}{Definition}
\newtheorem{itremark}{Remark}
\newtheorem{itclaim}{Claim}
\newtheorem{itfact}{Fact}
\newtheorem{itconjecture}{Conjecture}
\newtheorem{itcorollary}{Corollary}

\newenvironment{theorem}{\addtocounter{enunciato}{1}
\begin{ittheorem}}{\end{ittheorem}}

\newenvironment{lemma}{\addtocounter{enunciato}{1}
\begin{itlemma}}{\end{itlemma}}

\newenvironment{proposition}{\addtocounter{enunciato}{1}
\begin{itproposition}}{\end{itproposition}}

\newenvironment{definition}{\addtocounter{enunciato}{1}
\begin{itdefinition}}{\end{itdefinition}}

\newenvironment{remark}{\addtocounter{enunciato}{1}
\begin{itremark}}{\end{itremark}}

\newenvironment{conjecture}{\addtocounter{enunciato}{1}
\begin{itconjecture}}{\end{itconjecture}}

\newenvironment{corollary}{\addtocounter{enunciato}{1}
\begin{itcorollary}}{\end{itcorollary}}

\newcommand{\be}[1]{\begin{equation}\label{#1}}
\newcommand{\ee}{\end{equation}}

\newcommand{\bl}[1]{\begin{lemma}\label{#1}}
\newcommand{\el}{\end{lemma}}

\newcommand{\br}[1]{\begin{remark}\label{#1}}
\newcommand{\er}{\end{remark}}

\newcommand{\bt}[1]{\begin{theorem}\label{#1}}
\newcommand{\et}{\end{theorem}}

\newcommand{\bd}[1]{\begin{definition}\label{#1}}
\newcommand{\ed}{\end{definition}}

\newcommand{\bp}[1]{\begin{proposition}\label{#1}}
\newcommand{\ep}{\end{proposition}}

\newcommand{\bc}[1]{\begin{corollary}\label{#1}}
\newcommand{\ec}{\end{corollary}}

\newcommand{\bcj}[1]{\begin{conjecture}\label{#1}}
\newcommand{\ecj}{\end{conjecture}}

\newcommand{\hN}{\widehat{N}}

\renewcommand{\mid}{\,\middle|\,}
\newcommand{\ind}{\mathbbm{1}}

\def \Z {{\mathbb Z}}
\def \R {{\mathbb R}}

\def \N {{\mathbb N}}

\def \ba {\begin{array}}
\def \ea {\end{array}}

\def \P  {{\mathbb P}}
\def \E  {{\mathbb E}}

\def \di {\mathrm{d}}

\begin{document}

\title{Transient random walk in symmetric exclusion:\\ limit theorems and an Einstein relation.}

\author{\renewcommand{\thefootnote}{\arabic{footnote}}
L.\ Avena \footnotemark[1]
\\
\renewcommand{\thefootnote}{\arabic{footnote}}
R.\ dos Santos \footnotemark[2]
\\
\renewcommand{\thefootnote}{\arabic{footnote}}
F.\ V\"ollering \footnotemark[2]}

\footnotetext[1]{
Institut f\"ur Mathematik, Universit\"at Z\"urich, Winterthurerstrasse 190, Z\"urich, CH- 8057,
Switzerland}

\footnotetext[2]{
Mathematical Institute, Leiden University, P.O.\ Box 9512,
2300 RA Leiden, The Netherlands}

\maketitle

\begin{abstract}
We consider a one-dimensional simple symmetric exclusion process
in equilibrium, constituting a dynamic random environment for a
nearest-neighbor random walk that on occupied/vacant sites has two
different local drifts to the right. We construct a renewal
structure from which a LLN, a functional CLT and large deviation bounds for the random walk 
under the annealed measure follow. We further prove an Einstein relation under a
suitable perturbation. A brief discussion on the topic of random walks in slowly mixing
dynamic random environments is presented.\\

\vspace{0.1cm}\noindent
{\it Acknowledgement.} The authors are grateful to Frank den Hollander and Vladas Sidoravicius
for fruitful discussions.

\vspace{0.2cm}\noindent
\end{abstract}


\section{Introduction: model, results and motivation}
\label{sec:Introduction}


\subsection{Model}
\label{subsec:model}

Let \be{IPS}\xi = (\xi_t)_{t \geq 0} \quad \mbox{ with } \quad
\xi_t = \big(\xi_t(x)\big)_{x\in\Z} \ee be a c\`adl\`ag Markov
process with state space $\Omega=\{0,1\}^\Z$. We interpret the
states of $\xi$ by saying that at time $t$ the site $x$ is
\emph{occupied by a particle} if $\xi_t(x)=1$ and is \emph{vacant}
or, alternatively, \emph{occupied by a hole}, if $\xi_t(x)=0$. For
an initial configuration $\eta\in\Omega$, we write $P^\eta$ to
denote the law of $\xi$ starting from $\xi_0=\eta$, which is a
probability measure on the path space $D_\Omega[0,\infty)$, i.e.
the set of all trajectories with values in $\Omega$ which are
right continuous and have left limits, see \cite{Li85}, Section
I.1. We denote by \be{SEPlaw} P^{\mu}(\cdot) = \int_{\Omega}
P^\eta(\cdot)\,\mu(\di\eta) \quad \text{ on }D_\Omega[0,\infty)
\ee the law of $\xi$ when $\xi_0$ is drawn from a probability
measure $\mu$ on $\Omega$.

Having fixed a realization of $\xi$, let \be{X} X = (X_t)_{t\geq
0} \ee be the Random Walk (RW) that starts from $0$ and has local
transition rates \be{rwtrans}
\begin{aligned}
&x \to x+1 \quad \mbox{ at rate } \quad \alpha_1\,\xi_t(x) + \alpha_0\,[1-\xi_t(x)],\\
&x \to x-1 \quad \mbox{ at rate } \quad \beta_1\,\xi_t(x) + \beta_0\,[1-\xi_t(x)],
\end{aligned}
\ee where \be{drift1}
\alpha_0,\alpha_1,\beta_0,\beta_1\in(0,\infty), \ee
i.e., on occupied (resp. vacant) sites the random walk jumps
to the right at rate $\alpha_1$ and to the left at rate $\beta_1$
(resp. $\alpha_0$ and $\beta_0$). 

We write $P_X^\xi$ to denote the law of $X$ when $\xi$ is fixed
and, for an initial measure $\mu$, \be{Pxmudef} \P_{\mu}(\cdot) =
\int_{D_\Omega[0,\infty)} P_X^\xi(\cdot)\,P^{\mu}(\di\xi)
\quad\text{ on }D_\Z[0,\infty) \ee to denote the law of $X$
averaged over $\xi$. We refer to $P_X^\xi$ as the \emph{quenched}
law and to $\P_{\mu}$ as the \emph{annealed} law. 

We are interested in studying the RW $X$ when $\xi$ is a
one-dimensional \emph{Simple Symmetric Exclusion Process} (SSEP), i.e., an
interacting particle system whose generator $L$ acts on a real
cylinder function $f$ as \be{Generator} (Lf)(\eta) = \sum_{
{x,y\in\Z} \atop {x \sim y} } \left[f(\eta^{xy})-f(\eta)\right],
\qquad \eta\in\Omega, \ee where the sum runs over unordered pairs
of neighboring sites in $\Z$, and $\eta^{xy}$ is the configuration
obtained from $\eta$ by interchanging the states at sites $x$ and
$y$. For any $\rho \in (0,1)$, the Bernoulli product measure with
density $\rho$, which we denote by $\nu_{\rho}$, is an ergodic
measure for the SSEP (see \cite{Li85}, Theorem VIII.1.44).

We will assume that \be{drift2} \alpha_0\wedge \alpha_1 - \beta_0
\vee \beta_1> 1. \ee Condition \eqref{drift2} implies that the
local drifts on occupied and vacant sites, $\alpha_1-\beta_1$ and
$\alpha_0-\beta_0$ respectively, are both bigger than $1$. Thus
the RW $X$ is not only transient (indeed, non-nestling), but travels faster than local
information can spread in the SSEP. This is a strong property which is key to our argument;
it allows us, roughly speaking, to overcome the slow mixing in time of the SSEP with the good mixing 
in space of the Bernoulli measure $\nu_\rho$, giving rise to a regenerative structure for the random walk.

\subsection{Results}
\label{subsec:results}

For all results below we assume \eqref{drift1} and \eqref{drift2}, and fix $\rho \in [0,1]$.

\bt{LLNforEP}{\bf (Law of large numbers)}

There exists $v\geq\alpha_0\wedge \alpha_1 - \beta_0 \vee \beta_1>1$ such that 
\be{LLN1forEP} \lim_{t\to\infty} \frac{X_t}{t} = v
\qquad \P_{\nu_\rho}\text{-a.s. and in } L^p \;\; \forall \;\; p
\ge 1. \ee
\et

\bt{LDforEP}{\bf(Annealed large deviations)}

For any $\epsilon > 0$,
\be{eq:expconc}
\limsup_{t \to \infty} t^{-1}\log \P_{\nu_{\rho}}(|X_t-tv| \ge t\epsilon ) < 0.
\ee
\et

\bt{CLTforEP}{\bf (Annealed functional central limit theorem)}

There exists $\sigma\in(0,\infty)$ such that, under $\P_{\nu_\rho}$, 
\be{fCLT}\left(\frac{X_{nt}-ntv}{\sqrt{n}} \right)_{t \ge 0} \Rightarrow \sigma B
\ee
where $B$ is a standard Brownian motion.
\et

For the next result, we interpret the model of Section \ref{subsec:model} as a
perturbation of a homogeneous RW. We regard the exclusion process as
an oscillating random field which interacts weakly with the RW, affecting its asymptotic speed. 
The Einstein relation then says that the rate of change of the speed when the interaction 
is very weak is given by the diffusion coefficient of the unperturbed walk.
This is a form of the fluctuation-dissipation theorem from statistical physics, 
which concerns the response of thermodynamical systems to small external perturbations, 
connecting it with spontaneous fluctuations of the system. 
See \cite{DeDe10, FeGoLe85, GaMaPi12, KoOl05} for more information.

\bt{ERforEP}{\bf (Einstein Relation)}\\
Fix $\alpha,\beta >0$ with $\alpha-\beta>1$. Let
$\lambda\in(0,\infty)$ be the perturbation strength, and fix 
interaction constants $F_0,F_1 \in \R$ with $F_0+F_1=1$. Let the
perturbed rates be given by: \be{Perturbation}
\begin{aligned}
&\alpha_0 = \alpha \exp\left\{ F_{0}\frac{\lambda}{1-\rho}+o(\lambda) \right\},
&&\beta_0=\beta \exp\left\{ -F_{0}\frac{\lambda}{1-\rho}+o(\lambda) \right\},\\
&\alpha_1 = \alpha \exp\left\{ F_{1}\frac{\lambda}{\rho}+o(\lambda) \right\},
&&\beta_1=\beta \exp\left\{ -F_{1}\frac{\lambda}{\rho}+o(\lambda) \right\}.
\end{aligned}
\ee
When $\lambda$ is small enough, \eqref{drift2} is satisfied.
For such $\lambda$, let $v(\lambda)$ be the speed as in \eqref{LLN1forEP}. 
Then
\be{ER} \lim_{\lambda\downarrow
0}\frac{v(\lambda)-v(0)}{\lambda}=\alpha+\beta. 
\ee 
\et

\vspace{0.3cm}
The rest of the paper is organized as follows. In Section
\ref{subsec:RWRE}, we present a brief introduction to RW in
\emph{static} and \emph{dynamic} Random Environment (RE), and in
Section \ref{subsec:slowmix} we discuss slowly mixing dynamic REs. 
In Section \ref{sec:construction}, we construct a particular version of our model. 
Section \ref{sec:regeneration} is the core of the paper; there we
develop a regeneration scheme that is used in Section~\ref{sec:limthms} 
to prove Theorems \ref{LLNforEP}--\ref{ERforEP}.

\subsection{Random walks in static and dynamic random environments}
\label{subsec:RWRE} Random Walks in Random Environments (RWRE) on the
integer lattice are RWs on $\Z^d$ evolving according to random
transition kernels, i.e., their transition probabilities depend on
a random field (\emph{static} case) or a random process
(\emph{dynamic} case) called RE.

RWs in \emph{static} REs have been an intensive research area
since the early 1970's (see e.g.\ \cite{So75}). One-dimensional
models are well understood. In particular, recurrence vs.\
transience criteria, laws of large numbers and central limit theorems
have been derived, as well as quenched and annealed large
deviation principles. In higher dimensions the state of the art is
more modest and many important questions still remain open. For an
overview of the results, we refer the reader to
\cite{Sz02a,Ze04,Ze06}.

RWs in \emph{dynamic} REs in dimension $d$ can be viewed as RWs in
\emph{static} REs in dimension $d+1$ by considering the time as an
additional dimension (see e.g.\ \cite{AvdHoRe11}). Therefore even
the one dimensional case is still far from being understood, in
particular when the RE has dependencies in space and time. Three
classes of dynamic REs have been studied in the literature so far:
\begin{itemize}
\item[(1)] \emph{Independent in time}: globally updated at each
unit of time (see e.g.
\cite{Be04,BoMiPe04,BoMiPe09,JoRaAg11,RaAgSe05,Yi09b}); \item[(2)]
\emph{Independent in space}: locally updated according to
single-site independent Markov chains (see e.g.
\cite{BaZe06,BoMiPe07,DoLi09}); \item[(3)] \emph{Dependent in
space and time}
(\cite{AvdHoRe11,AvdHoRe10,BrKu09,DoKeLi08, dHodSa12, dHodSaSi11, JoRaAg11, ReVo11}).\end{itemize} The
focus of these references is: Law of Large Numbers (LLN),
invariance principles and large deviations estimates. All papers
require additional assumptions on the RE, e.g. a weak influence on
the RW (i.e., the RW is a small perturbation of a homogeneous one)
or a strong decay of space-time correlations.
We refer the reader to \cite{AvdHoRe10,DoKeLi08} for further references.

\subsection{Slowly mixing dynamic random environments
and the exclusion process} \label{subsec:slowmix} 
In \cite{AvdHoRe11}, a strong LLN was proved for RWs on a class of
Interacting Particle Systems (IPS) satisfying a 
space-time mixing property called \emph{cone-mixing}. This
mixing property can be described as the requirement
that all the states of the IPS inside a linearly growing
space-time region (a space-time cone) depend weakly on the states
of the IPS inside a space plane far below the tip of the cone.
The proof of the LLN in \cite{AvdHoRe11} uses a regeneration-time
argument introduced in \cite{CoZe04} for \emph{static} RE, which
was adapted to \emph{dynamic} (space-time ergodic) REs satisfying the above described cone-mixing property.
However, many interesting examples, which we call \emph{slowly
mixing} dynamic REs, are not cone-mixing due to slow or non-uniform decay of
space-time correlations. Examples include the exclusion process (and other Kawasaki dynamics), 
Poissonian fields of independent simple random walks (and other zero-range processes) 
and the supercritical contact process.
In these systems, the decay of correlations is not uniform, which prevents the use
of regeneration strategies such as in \cite{CoZe04}, \cite{AvdHoRe11} and \cite{dHodSaSi11}. 
In two recent papers a LLN is obtained for RWs in slowly mixing REs: in \cite{dHoKeSi11} for the case of a high-density Poissonian field of independent simple random walks, and in \cite{dHodSa12} for a supercritical contact process.

It is worthwhile to investigate examples of slowly mixing dynamic REs, 
as significantly different behavior may occur
in comparison to fast-mixing REs such as cone-mixing REs. Indeed, in
\cite{AvdHoRe10,AvTh12} the case of a RW $X$ on the
one-dimensional SSEP process with opposite drifts on top of
particles and holes (i.e.\ dropping \eqref{drift2} and assuming
$\alpha_1=\beta_0$, $\alpha_0=\beta_1$ in \eqref{drift1}) was
considered. In particular, in \cite{AvTh12}, simulation results
for the asymptotic speed of $X$ are presented which suggest that
$X$ is recurrent if and only if $\rho=\frac{1}{2}$, and that $X$
is ballistic as soon as it is transient. Thus, the transient
regime with zero speed, which is known to occur for static REs
(see e.g.\ \cite{So75}), seems to disappear in the dynamic setup.
The interpretation is that even `slow' particle motion in the RE
makes it hard for a `trap' to survive for too long. 
We recall that a `trap' is a localized region in which the walk 
spends a long time because the
transition probabilities push it towards the center of this
region. Nevertheless, similarly to the one-dimensional static RE
and in contrast to fast-mixing dynamic RE, Theorem 1.4 in
\cite{AvdHoRe10} shows that, when we look at large deviations
estimates for the empirical speed of $X$, the slow mixing of the 
SSEP process allows for a trap to persist up to
time $t$ with a probability that decays sub-exponentially in $t$.
Furthermore, other numerical results in \cite{AvTh12}
suggest non-diffusive scaling limits for $X$ in a certain
parameter region, as happens in the static case (see e.g.\
\cite{KeKoSp75,Si82}).

In the present paper, we take as dynamic random environment the SSEP, 
which is a natural example where mixing is both slow and non-uniform 
due to the conservation of particles. 
We study the RW under the strong drift assumption \eqref{drift2}, which significantly facilitates the analysis. 
Our results show that, in this case, the anomalous
behavior of fluctuations present in the static setting disappears; this is expected since such behavior 
is connected to trapping phenomena which are hindered by a positive minimum drift.
We believe that the regeneration strategy developed in Section \ref{sec:regeneration} 
could be adapted to other dynamic REs (for instance, asymmetric exclusion
processes or a Poissonian field of independent RWs) under
similar drift assumptions.

\section{Construction of the model}
\label{sec:construction} 

In this section, we give a particular construction of the random walk and 
of the exclusion process. With this construction, we introduce in Section \ref{subsec:minwalk} the notion of \emph{marked agents} and obtain as a consequence Lemma \ref{FreshBernoulli}, which plays a key role throughout the paper.

\newpage
\subsection{Coupling with the minimal walker} \label{subsec:minwalk}
Here we show how the RW $X$ defined in \eqref{X} can be constructed from
four independent Poisson processes and the RE. The following
construction is valid for any general dynamic RE given by a
two-state IPS.

Define the following set of Poissonian clocks, each independent of all
the other variables:
\be{Poissonclocks}
\begin{array}{lll}
N^{+}=(N_t^+)_{t\geq0} & \text{ with rate } & \alpha_0\wedge\alpha_1, \\
N^{-}=(N_t^-)_{t\geq0} & \text{ with rate } & \beta_0\wedge\beta_1, \\
\hN^{+}=(\hN_t^+)_{t\geq0} & \text{ with rate } & \alpha_0\vee\alpha_1-\alpha_0\wedge\alpha_1, \\
\hN^{-}=(\hN_t^-)_{t\geq0} & \text{ with rate } & \beta_0\vee\beta_1-\beta_0\wedge\beta_1.
\end{array}
\ee

Now define $X$ by the following rules:
\begin{enumerate}
\item $X$ jumps only when one of the Poisson clocks ring;
\item When $N^+$ rings, $X$ jumps to the right; when $N^-$
rings, $X$ jumps to the left;
\item When $\hN^+$ rings, $X$
jumps to the right if the state $j$ at its position is such that $\alpha_j = \alpha_0 \vee \alpha_1$. When
$\hN^-$ rings, $X$ jumps to the left if $\beta_j = \beta_0 \vee
\beta_1$. Otherwise, $X$ stays still.
\end{enumerate}
In this construction, $X$ is a function of $(N^\pm,\hN^\pm,\xi)$ and depends on the environment only through the states it sees when $\hN^+$ or $\hN^-$ ring. Let $M = (M_t)_{t \ge 0}$ be defined by
\be{defminimalwalker}
M_t := N_t^{+} -N_t^- - \hN_t^-.
\ee
By construction, for any $t \ge s \ge 0$,
\be{MStocDom}
M_t-M_s\leq X_t - X_s, \ee
and we are thus justified to call $M$ the \emph{minimal walker}.

Let
\be{Nt}
N_t:=N^+_t+N^-_t+\hN^+_t+\hN^-_t
\ee
be the number of attempted jumps before time $t$ and
\be{Nt2}
\hN_t := \hN_t^+ +\hN_t^-
\ee
the number of times before time $t$ when the random walk observes the environment.
Note that, by construction,
\be{dombym} |X_t - X_s| \le N_t - N_s \;\;\; \forall \;\; t \ge s \ge 0. \ee
As a consequence, for all $p \ge 1$, there is a $C(p) \in
(0,\infty)$ such that \be{Lpbound}
\sup_{\eta\in\Omega}\E_{\eta}[|X_t|^p] \le C(p) t^p. \ee Therefore,
by uniform integrability, as soon as a LLN holds, convergence in $L^p$, $p \ge 1$, will follow as well.


\subsection{Graphical representation: SSEP from the interchange process}
\label{subsec:graphrep}

The SSEP can be constructed from a graphical representation as follows.
Let
\be{Poisson} I=\left(I(x) \right)_{x\in\Z} \ee
be a collection of i.i.d. Poisson processes with rate $1$.
Draw the events of $I(x)$ on $\Z \times [0,\infty)$ as arrows
between the points $x$ and $x+1$. Then, for each $t>0$ and $x \in \Z$,
there exists (a.s.) a unique path in $\Z \times [0,\infty)$ starting at $(x,t)$ and ending in $\Z \times \{0\}$
going downwards in time but forced to cross any arrows it encounters; see Figure \ref{Interchange}.
Denote by $\gamma_t(x)\in\Z$ the end position of this path. The process $\gamma = (\gamma_t)_{t \ge 0}$ is called the \emph{interchange process}.
\begin{figure}[hbtp]
\vspace{1cm}
\begin{center}
\setlength{\unitlength}{0.3cm}
\begin{picture}(20,10)(0,0)
\put(0,0){\line(22,0){22}} \put(0,11){\line(22,0){22}}

\put(2,0){\line(0,12){12}}
\put(5,0){\line(0,12){12}} \put(8,0){\line(0,12){12}}
\put(11,0){\line(0,12){12}} \put(14,0){\line(0,12){12}}
\put(17,0){\line(0,12){12}} \put(20,0){\line(0,12){12}}

\qbezier[15](2.1,4)(3.5,4)(4.9,4)
\qbezier[15](5.1,6)(6.5,6)(7.9,6)
\qbezier[15](8.1,8.5)(9.5,8.5)(10.9,8.5)
\qbezier[15](11.1,3)(12.5,3)(13.9,3)
\qbezier[15](11.1,6)(12.5,6)(13.9,6)
\qbezier[15](14.1,1.5)(15.5,1.5)(16.9,1.5)
\qbezier[15](14.1,9)(15.5,9)(16.9,9)
\qbezier[15](17.1,6.5)(18.5,6.5)(19.9,6.5)

{\thicklines \qbezier(11,0)(11,3)(11,3)
\qbezier(11,3)(14,3)(14,3)
\qbezier(14,3)(14,6)(14,6)
\qbezier(14,6)(11,6)(11,6)
\qbezier(11,6)(11,8.5)(11,8.5)
\qbezier(11,8.5)(8,8.5)(8,8.5)
\qbezier(8,8.5)(8,11)(8,11)}


\put(2.8,4.2){$\leftrightarrow$}
\put(5.8,6.2){$\leftrightarrow$}
\put(8.8,8.7){$\leftrightarrow$}
\put(11.8,3.2){$\leftrightarrow$}
\put(11.8,6.2){$\leftrightarrow$}
\put(14.8,1.7){$\leftrightarrow$}
\put(14.8,9.2){$\leftrightarrow$}
\put(17.8,6.7){$\leftrightarrow$}

\put(10.4,-1.2){$\gamma_t(x)$} \put(8.4,11.6){$x$}
\put(-1.2,-.3){$0$}
\put(-1.2,10.7){$t$}\put(23,0){$\mathbb{Z}$}
\put(11,0){\circle*{.35}}
\put(8,11){\circle*{.35}}
\end{picture}
\end{center}
\caption{\small Graphical representation. The dotted lines represent
events of $I$. The thick lines mark
the path of the agent $\gamma_t(x)$.} \label{Interchange}
\end{figure}
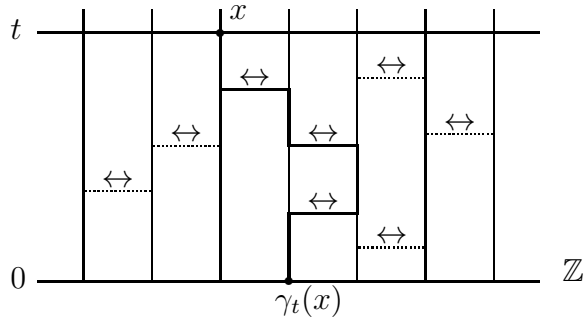
On the other hand, for each $t \ge 0$ and $x \in \Z$ there is a unique $y$ in $\Z$ such that $\gamma_t(y) = x$;
denote by
\be{inversegamma} \gamma^{-1} =
(\gamma^{-1}_t)_{t \ge 0}\ee
the process such that $\gamma^{-1}_t(x) = y$.

We interpret these processes by saying that there are \emph{agents} on the lattice,
named after their initial positions, who move around by exchanging places with their
neighbors at events of $I$. Then $\gamma^{-1}_t(x)$ is the position at time $t$ of
agent $x$ and $\gamma_t(x)$ is the agent who at time $t$ is at position $x$.

The SSEP $\xi=(\xi_t)_{t\geq 0}$ starting from a configuration
$\eta \in \Omega=\{0,1\}^\Z$ is obtained from $\gamma$ by
putting
\be{defSSE}
\xi_t(x) := \eta(\gamma_t(x)), \quad x \in \Z.
\ee
The description under the `agent interpretation' is that
we assign at time $0$ to each agent $x$ a state $\eta(x)$ and declare the state of
the exclusion process at a space time position $(x,t)$ to be the state of the
agent who is there.

We will call $\widetilde P$ the joint law of
$(N^+,N^-,\hN^+,\hN^-,I)$. For simplicity of notation, we
redefine $\P_{\mu}$ as the joint law of
$(N^+,N^-,\hN^+,\hN^-,I)$ and $\eta$ when the latter is distributed as $\mu$,
 i.e., $\P_{\mu} = \mu \times \widetilde{P}$.
 Then $\xi$ as defined in \eqref{defSSE} is under $\P_{\mu}$ indeed distributed
 as a SSEP started from $\mu$.


\subsection{Marked agents set}
\label{subsec:markedagents}
In our proof, regeneration arises as a consequence of the fact that, even though the environment
is slowly mixing, the environment \emph{perceived} by the walker is fast mixing in some sense.
The idea is that, since $X$ has a strong drift and the information spread is limited,
the dependence on the observed environment is left behind very fast. In the exclusion process,
this dependence is carried by the agents of the interchange process whom the RW $X$ meets as it moves;
we will therefore keep track of them via the following time-increasing set of \emph{marked agents}:
\be{MarkedSet}
A_t := \bigcup_{\substack{0 < s\leq t \\ \hN_{s-}\neq \hN_{s} }}\left\{\gamma_s(X_{s-})\right\}.
\ee
In words, $A_t$ consists of all the agents $x \in \Z$ whose states the walker observes up to time $t$. 
Set also
\be{supMarkedSet}
R_t := \sup_{x \in A_t} \gamma^{-1}_t(x),
\ee
i.e., $R_t$ is the position of the rightmost marked agent at time $t$. 
As usual we take $\sup \emptyset = - \infty$.

An important observation is that the walker depends on the initial configuration only through the states of the agents in $A_t$.
More precisely, $X$ is adapted to the filtration $\mathcal{G} = (\mathcal{G}_t)_{t \ge 0}$ given by
\be{filtrationG}
\mathcal{G}_t := \sigma((N_s^{\pm},\hN_s^{\pm},I_s)_{0 \le s \le t}, A_t, (\eta(x))_{x \in A_t}).
\ee
Moreover, by the i.i.d. structure and exchangeability of $\nu_\rho$,
the states of the agents who are not in $A_t$ have still, given $\mathcal{G}_t$, distribution $\nu_\rho$.
This is the content of the following lemma.

\bl{FreshBernoulli} For any $t \ge 0$ and $x_1, \ldots x_n \in \Z$,
\be{FreshBer}\E_{\nu_\rho}\left[\prod_{i=1}^{n}\xi_t(x_i) \mid
\mathcal{G}_t \right] = \rho^n \; \text{ a.s. on }
\{\gamma_t(x_1),...,\gamma_t(x_n) \notin A_t\},
\ee
i.e., the SSEP at time $t$ and off $\gamma^{-1}_t(A_t)$
 is, given $\mathcal{G}_t$, distributed according to $\nu_\rho$.
Moreover, \eqref{FreshBer} is still valid when $t$ is replaced with
a finite $\mathcal{G}$-stopping time.
\el
\begin{proof} From the definition of $A_t$ it follows that,
for $A \subset \Z$, \be{fb1}
\{A_t = A\} \in \sigma((N_s^{\pm},\hN_s^{\pm},I_s)_{0 \le s \le t},
(\eta(x))_{x \in A}). \ee With \eqref{fb1} we can verify by
summing over $A$ that, for any $x_1,...,x_n \in \Z$,
\be{fb2}
\E_{\nu_{\rho}}\left[\prod_{i=1}^{n}\eta(x_i) \mid
\mathcal{G}_t \right] = \rho^n \; \text{ a.s. on the set }
\{x_1,...,x_n \notin A_t\}. \ee
The summation is justified because $A_t$ is, for each $t$, a finite set.
Since $\gamma$ is $\mathcal{G}$-adapted and $\xi_t(x) =
\eta(\gamma_t(x))$, \eqref{FreshBer} follows.
The extension to a $\mathcal{G}$-stopping time is done by
approximating it from above by stopping times taking
values in a countable set (to which \eqref{FreshBer} easily extends)
and then using the right-continuity of $A_t$ and $\xi_t$.
\end{proof}

\section{Regeneration}
\label{sec:regeneration}
In this section we will develop a regenerative structure for the path of the RW $X$.
Let us first give an informal description of the regeneration strategy. Since $X$ is travelling
fast to the right, there will be moments, called \emph{trial times},
when the RW has left behind all agents previously met. At these times, it may `try to regenerate',
and we say that it succeeds if afterwards it never meets those agents again. In case it does not succeed,
we wait for the moment when it meets an agent from the past, which we call a \emph{failure time},
and repeat the procedure by waiting for the next trial time. Summarizing, the regeneration strategy consists of two steps:
waiting for a trial time when there is a chance for the walker to forget its past,
and then checking whether it succeeds or fails in its regeneration attempt. These steps are repeated
until the walker succeeds, which will eventually happen by the strong drift assumption \eqref{drift2}.

We proceed to formalize the regeneration scheme, beginning with the trial times.
Let $(T_t)_{t \ge 0}$ be the family of $\mathcal{G}$-stopping times defined by:
\be{Regeneration}
T_t:=
\inf\Big\{s \geq J_t
\colon\, X_s> R_s \Big\}.
\ee
where $J_t:=\inf\{s\geq t: N_t\neq N_s\}$ is the time of the next possible jump after time $t$.
The previous discussion justifies calling $T_t$ the first \emph{trial time} after time $t$.
From the definition it is clear that they are indeed $\mathcal{G}$-stopping times.
Note that, a.s., $T_t > t$.

In order to define the failure times, first let, for $t \ge 0, x\in\Z$,
\be{PRP} Y^t(x)=(Y^t_s(x))_{s \ge t} \ee
be the path starting at time $t$ from $x$ and jumping
to the right across the arrows of the process $I$ in
\eqref{Poisson}; see Figure \ref{RightPoisson2}.
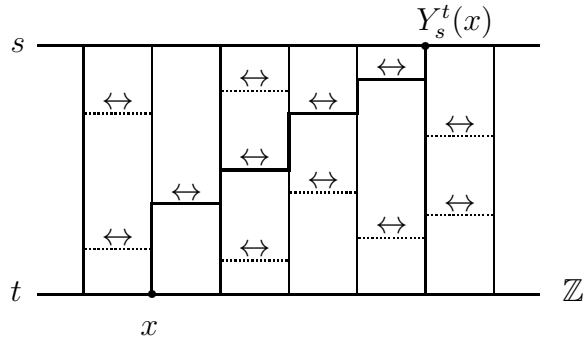
\begin{figure}[hbtp]
\vspace{1cm}
\begin{center}
\setlength{\unitlength}{0.3cm}
\begin{picture}(20,10)(0,0)

\put(0,0){\line(1,0){22}}
\put(0,11){\line(1,0){22}}

\put(2,0){\line(0,1){11}}
\put(5,0){\line(0,1){11}}
\put(8,0){\line(0,1){11}}
\put(11,0){\line(0,1){11}}
\put(14,0){\line(0,1){11}}
\put(17,0){\line(0,1){11}}
\put(20,0){\line(0,1){11}}

\qbezier[15](2.1,2)(3.5,2)(4.9,2)
\qbezier[15](2.1,8)(3.5,8)(4.9,8)
\qbezier[15](8.1,1.5)(9.5,1.5)(10.9,1.5)
\qbezier[15](8.1,9)(9.5,9)(10.9,9)
\qbezier[15](11.1,4.5)(12.5,4.5)(13.9,4.5)
\qbezier[15](14.1,2.5)(15.5,2.5)(16.9,2.5)
\qbezier[15](17.1,3.5)(18.5,3.5)(19.9,3.5)
\qbezier[15](17.1,3.5)(18.5,3.5)(19.9,3.5)
\qbezier[15](17.1,7)(18.5,7)(19.9,7)

\put(2.8,2.2) {$\leftrightarrow$}
\put(2.8,8.2) {$\leftrightarrow$}
\put(5.8,4.2) {$\leftrightarrow$}
\put(8.8,1.7) {$\leftrightarrow$}
\put(8.8,5.7) {$\leftrightarrow$}
\put(8.8,9.2) {$\leftrightarrow$}
\put(11.8,4.7) {$\leftrightarrow$}
\put(11.8,8.2) {$\leftrightarrow$}
\put(14.8,2.7){$\leftrightarrow$}
\put(14.8,9.7){$\leftrightarrow$}
\put(17.8,3.7) {$\leftrightarrow$}
\put(17.8,7.2) {$\leftrightarrow$}

{\thicklines
\qbezier(5,0)(5,4)(5,4)
\qbezier(5,4)(8,4)(8,4)
\qbezier(8,4)(8,5.5)(8,5.5)
\qbezier(8,5.5)(11,5.5)(11,5.5)
\qbezier(11,5.5)(11,8)(11,8)
\qbezier(11,8)(14,8)(14,8)
\qbezier(14,8)(14,9)(14,9.5)
\qbezier(14,9.5)(17,9.5)(17,9.5)
\qbezier(17,9.5)(17,11)(17,11)}

\put(-1.2,-.3){$t$}\put(-1.2,10.7){$s$} 
\put(4.5,-1.8){$x$} \put(16.6,11.7){$Y^t_s(x)$}
\put(5,0){\circle*{.35}} 
\put(17,11){\circle*{.35}} \put(23,-.3){$\Z$}
\end{picture}
\end{center}
\caption{\small As in Figure \ref{Interchange}, the dotted lines
are events of $I$. The path $Y^t(x)$ starts at
$x$ and goes upwards in time and to the right across the arrows.} \label{RightPoisson2}
\end{figure}
Then $(Y^t_{t+u}(x) - x)_{u \ge 0}$ is a Poisson process with rate $1$.

Now let $(F_t)_{t \ge 0}$ be the family of $\mathcal{G}$-stopping times defined by
\be{defFt}
F_t := \inf\{s > t \colon X_s \le Y^t_s(X_t-1)\}. \ee
As usual we take $\inf \emptyset = \infty$.
We call $F_t$ the first \emph{failure time} after time $t$.
The $F_t$'s are smaller than the failure times informally discussed in the beginning of the section.
Indeed, agents to the left of $X_t$ at time $t$ can never cross $Y^t(X_t-1)$,
as can be seen on the graphical representation. In particular, if
$F_t = \infty$, then $X$ will after time $t$ never meet such agents again.

In the following lemma we obtain exponential moment bounds for the trial times $T_t$, showing in particular that they are a.s. finite.

\bl{momTt} For every $a > 0$, there exists $b_1 \in (0,\infty)$ such that, for all $t \ge 0$,
\be{eq:momTt}
\E_{\nu_\rho}[e^{b_1 (T_t-t)} | \mathcal{G}_t ] \le \left(1+a\right)e^{a (R_t-X_t)^+} \quad \P_{\nu_\rho}\text{-a.s.}
\ee
\el
\begin{proof}
Let \be{Y}\widetilde{Y}^t=Y^t(R_t \vee X_t)\ee
be the Poisson path starting at time $t$ from the position $R_t \vee X_t$.

Define $H_t := \inf\{s > t \colon\, M_s-M_t + X_t > \widetilde{Y}^t_s\}$. Let us check that
\be{CrossTime}
T_t\leq H_t \vee J_t. \ee
Indeed, if $X_{J_t} > \widetilde{Y}^t_{J_t}$ (which can happen only if $R_t \le X_t$), then $T_t = J_t$.
Suppose now that $X_{J_t} \le \widetilde{Y}^t_{J_t}$.
Recall the definition of $\gamma^{-1}$ in \eqref{inversegamma}. 
By geometrical constraints, if $\gamma^{-1}_s(x) \le \widetilde{Y}^t_s$ for some $s
\ge t$, then this will also hold for all future times.
In particular, agents marked by $X$ before it crosses $\widetilde{Y}^t$
will never be able to cross $\widetilde{Y}^t$ themselves. This implies that $T_t$ is smaller
than the first time after $t$ when $X$ is to the right of $\widetilde{Y}^t$, which is
in turn smaller than $H_t$ by \eqref{MStocDom}.

Since the minimal walker $M$ is
independent of $I$, $\left(M_{t+u}-M_t - (\widetilde{Y}^t_{t+u} - R_t \vee X_t) \right)_{u \ge 0}$
is under $\P_{\nu\rho}(\cdot | \mathcal{G}_t)$ a continuous-time RW starting from $0$ that
jumps to the right at rate $\alpha_0\wedge\alpha_1$ and to the left at rate
$\beta_0\vee\beta_1+1$. By \eqref{drift2}, this RW is transient to the right with speed $\alpha_0\wedge\alpha_1-\beta_0\vee\beta_1-1>0$.
Furthermore, $H_t-t$ is the first time when it hits $(R_t - X_t)^+ + 1$.
Now, if $\mathcal{T}_x$ is the first time when a continuous-time RW with drift $d > 0$ hits a site $x>0$,
then $\sup_{x \ge 1}\left(\mathcal{T}_x- 2x/d \right)^+$ has an exponential moment, which can be taken arbitrarily close to $1$. 
Therefore, by \eqref{CrossTime}, \eqref{eq:momTt} holds for $b_1$ sufficiently small.
\end{proof}

For $t \ge 0$, denote by $X^{(t)}$ the increments of the walk after time $t$, that is,
\be{defXt}
X^{(t)}_u := X_{t+u}-X_t.
\ee

The next lemma shows that the second step of the regeneration strategy indeed works.

\bl{condgamma}
For each $t \ge 0$,
\be{eqcondgamma}
\P_{\nu_\rho}\left(F_t = \infty, X^{(t)} \in \cdot \mid \mathcal{G}_t\right) = \P_{\nu_\rho}\left( \Gamma, X \in \cdot \right) \text{ a.s. on } \{R_t < X_t\},
\ee
where $\Gamma := \{F_0 = \infty\}$.
\el
\begin{proof}
First note that
\be{indepleftongamma}
\eta \mapsto \P_{\eta}\left(\Gamma, X \in \cdot \right) \text{ does not depend on } (\eta(x))_{x < 0}.
\ee
This can be verified using the graphical representation.
Indeed, the agents $x < 0$ can never cross $Y^{0}(-1)$. Therefore, on $\Gamma$, none of them ever meets $X$, i.e., $A_t \cap \left(\Z \setminus \N_0 \right) = \emptyset$ for all $t$. On the other hand, $\Gamma$ is itself measurable in $\sigma(X,I)$; since $X$ is adapted to $\mathcal{G}$, \eqref{indepleftongamma} follows.

Now, letting $\bar{\xi}_t(\cdot) := \xi_t(X_t + \cdot)$, we can write
\be{compcondgamma}
\begin{array}{lcl}
\P_{\nu_\rho}\left( R_t < X_t, F_t = \infty, X^{(t)} \in \cdot \mid \mathcal{G}_t \right) & = & \E_{\nu_\rho} \left[\mathbbm{1}_{\{R_t < X_t\}} \P_{\bar{\xi}_t} \left(\Gamma, X \in \cdot \right) \mid \mathcal{G}_t \right] \\
& = & \mathbbm{1}_{\{R_t < X_t\}} \P_{\nu_\rho}\left( \Gamma, X \in \cdot \right),\\
\end{array}
\ee
where the first equality holds by the Markov property and translation-invariance of the graphical representation
and the second is justified since, by \eqref{indepleftongamma}, $\P_{\bar{\xi}_t} \left(\Gamma, X \in \cdot \right)$ is a function only of
$(\bar{\xi}_t(x))_{x \ge 0}$, whose distribution under $\P_{\nu_\rho}(\cdot | \mathcal{G}_t)$ is, by Lemma \ref{FreshBernoulli}, a.s. equal to $\nu_\rho$ when $R_t < X_t$.
\end{proof}

Before proceeding we make a simple but nonetheless important remark:
\begin{remark}\label{remarkstoptimes} \emph{Replacing $t$ in $T_t$ and $F_t$ with a finite $\mathcal{G}$-stopping time still yields a stopping time, and Lemmas \ref{momTt}--\ref{condgamma} (as well as Lemmas \ref{momFt} and \ref{momvarrho} below) remain true with a finite stopping time in place of $t$.}
\end{remark}
\vspace{-0.2cm}
Remark \ref{remarkstoptimes} is justified by right-continuity as in the proof of Lemma \ref{FreshBernoulli}. Recall also that a stopping time multiplied by the indicator function of the set where it is finite is again a stopping time.

\vspace{0.1cm}
We are now in shape to prove our main result.

\bt{FiniteFullReg} There exists a $\P_{\nu_{\rho}}$-a.s.\ positive and finite random time $\tau$ such that, $\P_{\nu_{\rho}}$-a.s.,
\begin{align}
&\P_{\nu_{\rho}}\Big( \left( X_{\tau+s}-X_\tau\right)_{s \ge 0} \in \cdot \,\Big|\, \tau, (X_s)_{s \le \tau} \Big) = \P_{\nu_{\rho}}\Big( X \in \cdot \,\Big|\, \Gamma \Big);   \label{proptau1}\\
 &\P_{\nu_{\rho}}\Big( \left( X_{\tau+s}-X_\tau\right)_{s \ge 0} \in \cdot \,\Big|\, \Gamma, \tau, (X_s)_{s \le \tau} \Big) = \P_{\nu_{\rho}}\Big( X \in \cdot \,\Big|\, \Gamma \Big) . \label{proptau2}
\end{align}
\et
\begin{proof}
We will obtain the regeneration time $\tau$ with the help of
an increasing sequence $(U_n)_{n \in \N_0}$ of
$\mathcal{G}$-stopping times in $[0,\infty]$,
which will be defined using $T_t$ and $F_t$.
We will throughout the proof tacitly use Remark \ref{remarkstoptimes}.

Set $U_0 := 0$. Supposing that for some $n \ge 0$, $(U_k)_{k \le 2n}$
are all defined, let
\be{defU}
\begin{array}{lcl}
U_{2n+1} & := & \left\{ \begin{array}{ll}
                   \infty & \;\;\; \text{if } U_{2n}=\infty \\
                   T_{U_{2n}} & \;\;\; \text{otherwise,} \\
                   \end{array}\right.\\
U_{2(n+1)} & := & \left\{ \begin{array}{ll}
                   \infty & \text{if } U_{2n+1}=\infty \\
                   F_{U_{2n+1}} & \text{otherwise.} \\
                   \end{array}\right.\\
\end{array}
\ee
Then $(U_n)_{n \in \N_0}$ is an increasing sequence of $\mathcal{G}$-stopping times.
Now define
\be{defK}
K = \inf\{n \in \N_0 \colon U_{2n+1} < \infty, F_{U_{2n+1}} = \infty \} \in [0,\infty],
\ee
i.e., $2K+1$ is the first index before the sequence $U$ hits infinity.

Set $\kappa := \P_{\nu_\rho}(\Gamma)$. Then $\kappa > 0$ since $X$ dominates $M$ and $M-Y^0(-1)$ has a positive drift.
By Lemma \ref{condgamma},
\be{tailK}
\P_{\nu_{\rho}}\left(K \ge n\right) = (1-\kappa)^{n} \;\; \forall \; n \in \N_0.
\ee
In particular, $K < \infty$ $\; \P_{\nu_\rho}$-a.s. and we can define
\be{deftau}
\tau := U_{2K+1} < \infty \quad \P_{\nu_{\rho}} \text{-a.s.}
\ee

Since $\P_{\nu_\rho}(\cdot | \Gamma) \ll \P_{\nu_\rho}$,
$\tau$ is a.s. well-defined and finite also under $\P_{\nu_\rho}(\cdot | \Gamma)$.

We will now proceed to verify \eqref{proptau1}.
Define $\mathcal{G}_{\tau}$ as the sigma-algebra of the events $B$ such that, for all $n \in \N_0$,
there exist $B_n \in \mathcal{G}_{U_{2n+1}}$ such that $B \cap \{K=n\} = B_n \cap \{K=n\}$.
Note that $\tau$ and $(X_s)_{s \le \tau}$ are measurable in $\mathcal{G}_{\tau}$.

Take $f \ge 0$ measurable, $B \in \mathcal{G}_{\tau}$, and write
\begin{align*}
&\E_{\nu_{\rho}}\left[\mathbbm{1}_{B} f(X^{(\tau)}) \right] = \sum_{n=0}^{\infty} \E_{\nu_{\rho}}\left[\mathbbm{1}_{B_n} \mathbbm{1}_{\{K=n\}} f(X^{(U_{2n+1})}) \right] \\
 &= \sum_{n=0}^{\infty} \E_{\nu_{\rho}}\left[\mathbbm{1}_{B_n}\mathbbm{1}_{\{U_{2n+1} < \infty, F_{U_{2n+1}}  = \infty\}} f(X^{(U_{2n+1})}) \right] \\
 &= \sum_{n=0}^{\infty} \E_{\nu_{\rho}}\left[\mathbbm{1}_{B_n} \mathbbm{1}_{\{U_{2n+1} < \infty\}} \E_{\nu_{\rho}}\left[\mathbbm{1}_{\{F_{U_{2n+1}} = \infty\}} f(X^{(U_{2n+1})}) \,\middle|\, \mathcal{G}_{U_{2n+1}}\right] \right].
\end{align*}
When $U_{2n+1} < \infty$, $R_{U_{2n+1}} < X_{U_{2n+1}}$ so, by Lemma \ref{condgamma}, the last line equals
\begin{align*}
 & \E_{\nu_{\rho}}\left[ f(X) \mathbbm{1}_{\Gamma} \right] \sum_{n=0}^{\infty} \E_{\nu_{\rho}}\left[\mathbbm{1}_{B_n} \mathbbm{1}_{\{U_{2n+1}<\infty\}} \right] \\
 &= \E_{\nu_{\rho}}\left[ f(X) \mid \Gamma \right] \sum_{n=0}^{\infty} \E_{\nu_{\rho}}\left[\mathbbm{1}_{B_n} \mathbbm{1}_{\{U_{2n+1}<\infty\}} \right] \P_{\nu_\rho}(\Gamma)
\end{align*}
which, by Lemma \ref{condgamma} again, is equal to
\begin{align}
 & \E_{\nu_{\rho}}\left[ f(X) \,\middle|\, \Gamma \right] \sum_{n=0}^{\infty} \E_{\nu_{\rho}}\left[\mathbbm{1}_{B_n} \mathbbm{1}_{\{U_{2n+1}<\infty\}} \P_{\nu_{\rho}}\left(F_{U_{2n+1}} = \infty \,\middle|\, \mathcal{G}_{U_{2n+1}}\right) \right] \nonumber\\
 &= \E_{\nu_{\rho}}\left[ f(X) \mid \Gamma \right] \sum_{n=0}^{\infty}\P_{\nu_{\rho}}\left(B_n, K=n \right) \nonumber\\
 &= \E_{\nu_{\rho}}\left[ f(X) \mid \Gamma \right] \P_{\nu_{\rho}}(B). \label{compproptau1}
\end{align}
This proves \eqref{proptau1}. 
To finish the proof, note that $\Gamma \in \mathcal{G}_{\tau}$ since, for any $t \ge 0$,
\be{gammainGtau}
\Gamma \cap \{F_t = \infty \}= \{X_s > Y^0_s(-1) \; \forall \; s \le t\} \cap \{F_t = \infty\}.
\ee
So \eqref{proptau2} follows by applying \eqref{compproptau1} to $B \cap \Gamma$ in place of $B$.
\end{proof}

In Proposition \ref{momtau} below, we will show that $\tau$ and $X_{\tau}$
have exponential moments. For its proof, we will need the following two lemmas. 
\bl{momFt}
For all $\epsilon > 0$, there exists $a_1 \in (0,\infty)$ such that, for all $t \ge 0$,
\be{eq:momFt}
\E_{\nu_\rho}\left[ \mathbbm{1}_{\{F_t < \infty\}} e^{a_1 (F_t - t)} \mid \mathcal{G}_t \right] \le 1+\epsilon \qquad \P_{\nu_\rho} \text{-a.s.}
\ee
\el
\begin{proof}
Let
\be{lasttimeleftY}
D_t := \sup \{s > t ; M_s - M_t + X_t \le Y^t_s(X_t-1)\}.
\ee
If $F_t < \infty$, then $F_t \le D_t$ because, when finite,
$F_t$ is smaller than the last time $s>t$ when $X_s \le Y^t_s(X_t-1)$, which is in turn
smaller than $D_t$ by \eqref{MStocDom}.
On the other hand, $(M_{t+u} - M_t + X_t - Y^t_{t+u}(X_t-1))_{u \ge 0}$ is under $\P_{\nu_\rho}(\cdot | \mathcal{G}_t)$ a continuous-time RW with positive drift starting at $1$. Since $D_t-t$ is the last time
when this random walk is less or equal to $0$, \eqref{eq:momFt} follows.
\end{proof}

\bl{momvarrho} For all $\epsilon>0$, there exists $a_2 \in (0,\infty)$
such that, for all $t \ge 0$, \be{eq:momvarrho}
\E_{\nu_\rho}\left[\mathbbm{1}_{\{F_t < \infty\}}
e^{a_2(R_{F_t}-X_{F_t})^+} \mid \mathcal{G}_t\right] \le 1+\epsilon
\quad \P_{\nu_\rho} \text{-a.s. on } \{R_t < X_t \}. \ee
\el
\begin{proof}
Take $D_t$ as in \eqref{lasttimeleftY} and recall that, when finite, $F_t \le D_t$. 
Let $\chi_t := X_t + N_{D_t}-N_t$ and consider $Y^t(\chi_t)$ (see \eqref{PRP}).
If $R_t < X_t$, then $R_{F_t} \le Y^t_{F_t}(\chi_t)$
and so 
\be{boundvarrho} R_{F_t} -
X_{F_t} \le Y^t_{D_t}(\chi_t)-\chi_t + N_{D_t} -N_t + 1. 
\ee 
Now \eqref{eq:momvarrho} follows by noting that, even though $\chi_t$ is not in $\mathcal{G}_t$, it is independent of $(Y^t_{t+u}(\chi_t)-\chi_t)_{u \ge 0}$ (as they depend on disjoint regions of the graphical representation), so that the latter is still a Poisson process under $\P_{\nu_\rho}(\cdot | \mathcal{G}_t)$.
\end{proof}

\bp{momtau} There exists $b \in (0,\infty)$ such that
\be{eq:momtau} \E_{\nu_\rho}[e^{b \tau}] , \; \E_{\nu_\rho}[e^{b
N_{\tau}}] < \infty, \ee the same being true under
$\P_{\nu_\rho}(\cdot | \Gamma)$. \ep
\begin{proof}
The last sentence follows from \eqref{eq:momtau} and $\kappa =
\P_{\nu_\rho}(\Gamma) > 0$. Since $N$ is a Poisson process, it is
enough prove to that $\tau$ has exponential moments under
$\P_{\nu_\rho}$. To this end, let $\epsilon > 0$ such that
$(1+\epsilon)^2(1-\kappa) < 1$. Take $a \in (0,\epsilon)$ such
that, for all $t \ge 0$, \be{midestimate}
\E_{\nu_{\rho}}\left[\mathbbm{1}_{\{F_t < \infty\}}
e^{a(F_t-t)+a(R_{F_t}-X_{F_t})^+} \mid \mathcal{G}_t \right]
\le 1+\epsilon \quad \P_{\nu_\rho} \text{-a.s. on } \{R_t
< X_t\}. \ee Such $a$ exists by Lemmas \ref{momFt} and
\ref{momvarrho} and an application of H\"older's inequality.
For this $a$, take $b_1$ as in Lemma \ref{momTt} and let $b := (a \wedge b_1)/2$.
Now fix $n\ge1$ and estimate, recalling that $R_{U_{2n-1}} < X_{U_{2n-1}}$ when
$U_{2n-1} < \infty$,
\begin{align}\label{est1}
\E_{\nu_\rho} \left[ \mathbbm{1}_{\{U_{2n} < \infty\}}
e^{2bU_{2n+1}} \right]  = 
\E_{\nu_\rho} \left[ \mathbbm{1}_{\{U_{2n} < \infty\}} e^{2bU_{2n}}\E_{\nu_{\rho}}\left[ e^{2b(T_{U_{2n}}-U_{2n})} \mid \mathcal{G}_{U_{2n}} \right] \right]\hspace{3cm} \nonumber \\
 \le (1+a) \E_{\nu_\rho} \left[ \mathbbm{1}_{\{U_{2n} < \infty\}} e^{2bU_{2n}+a(R_{U_{2n}}-X_{U_{2n}})^+} \right] \hspace{3.35cm} \nonumber \\
  = (1+a) \E_{\nu_\rho} \Bigg\{ \mathbbm{1}_{\{U_{2n-2} < \infty\}}e^{2bU_{2n-1}} \; \hspace{8.5cm} \nonumber \\
        \times \; \E_{\nu_{\rho}} \left[\mathbbm{1}_{\{F_{U_{2n-1}} < \infty\}} e^{2b(F_{U_{2n-1}}-U_{2n-1})+a\left(R_{F_{U_{2n-1}}}-X_{F_{U_{2n-1}}}\right)^+} 
          \mid \mathcal{G}_{U_{2n-1}}\right] \Bigg\} \nonumber \\
 \le (1+\epsilon)^2 \E_{\nu_\rho} \left[ \mathbbm{1}_{\{U_{2(n-1)}
< \infty\}} e^{2bU_{2(n-1)+1}} \right]. \hspace{4.2cm}
\end{align}
By induction, we get \be{est2} \E_{\nu_\rho} \left[
\mathbbm{1}_{\{U_{2n} < \infty\}} e^{2bU_{2n+1}} \right] \le
(1+\epsilon)^{2n+1}. \ee 

To conclude, use H\"older's
inequality and \eqref{tailK} to write:
\begin{equation}\label{est3}
\begin{array}{ll}
\vspace{0.2cm}
\E_{\nu_\rho}\left[e^{b\tau} \right] & = \sum_{n=0}^{\infty} 
\E_{\nu_\rho}\left[\mathbbm{1}_{\{K=n\}}e^{bU_{2n+1}} \right] =
\sum_{n=0}^{\infty} \E_{\nu_\rho}\left[\mathbbm{1}_{\{K=n\}}\mathbbm{1}_{\{U_{2n}< \infty\}}e^{bU_{2n+1}} \right] \\
\vspace{0.2cm}
& \le \sum_{n=0}^{\infty} \P_{\nu_\rho}\left(K=n\right)^{\frac{1}{2}}
\E_{\nu_\rho}\left[\mathbbm{1}_{\{U_{2n}< \infty\}}e^{2b U_{2n+1}} \right]^{\frac{1}{2}} \\
& \le \sqrt{1+\epsilon} \sum_{n=0}^{\infty} \left(\sqrt{(1-\kappa)(1+\epsilon)^2}\right)^n < \infty.
\end{array}
\end{equation}
\end{proof}

Finally, due to Theorem \ref{FiniteFullReg}, we can construct a sequence of i.i.d. regeneration
times.

\bt{regtimes} By enlarging the probability space, one can assume the existence of a sequence $(\tau_n)_{n \in \N}$ of
random times with $\tau_1 := \tau$ and such that, setting $S_n:=\sum_{i=1}^n\tau_i$,
\be{iidsequence}
\Big(\tau_{n+1},\left(X^{(S_n)}_s\right)_{0 \le s \le
\tau_{n+1}}\Big)_{n \in \N}
\ee 
is under $\P_{\nu_\rho}$ an i.i.d. sequence which is independent from $(\tau,(X_s)_{0 \le s \le \tau})$, each of its terms being distributed as $(\tau,(X_s)_{0 \le s \le \tau})$ under $\P_{\nu_\rho}(\cdot \vert \Gamma)$.
\et
\begin{proof}
A version of $X$ with the claimed properties can be constructed on
a product space using Theorem \ref{FiniteFullReg}, as is standard
for ``delayed regenerative processes'' (see e.g. \cite{SiThWo94}).
This version can be assumed to be the one constructed in Section
\ref{subsec:minwalk} again by a standard coupling argument.
\end{proof}

\newpage
\section{Limit Theorems}
\label{sec:limthms} As a fruit of the regenerative structure constructed in Section~\ref{sec:regeneration}, we now obtain the asymptotic results stated in Section \ref{subsec:results}.

\subsection{Proofs of Theorems \ref{LLNforEP} --- \ref{CLTforEP}}
\label{subsec:LLN}

Let us collect some useful facts.
First of all, by Theorem \ref{regtimes}, Proposition \ref{momtau} and \eqref{dombym}, 
\be{sups}
\left(\sup_{s \in [0, \tau_{n+1}]} \left|X_s^{(S_n)} \right| \right)_{n \in \N_0} \text{ have a uniform exponential moment.}
\ee
Furthermore, again by Theorem \ref{regtimes} and Proposition \ref{momtau},
\be{LLN2}\lim_{n\to\infty} \frac{S_n}{n} = \E_{\nu_{\rho}}[\tau \vert \Gamma] \quad\text{ and }
\quad\lim_{n\to\infty} \frac{X_{S_n}}{n} = \E_{\nu_\rho}[X_{\tau} \vert \Gamma] \quad \P_{\nu_\rho} \text{-a.s.}\ee
For $t\geq0$, let $k_t$ be the random integer such that
\be{SubSequence} S_{k_t}\leq t < S_{k_t+1}.\ee 
Then a.s. $\lim_{t \to \infty}t^{-1}k_t=\E_{\nu_{\rho}}[\tau \vert \Gamma]^{-1}$.
Thus the candidate velocity for $X$ is
\be{defspeed} v := \frac{\E_{\nu_\rho}[X_{\tau} \vert \Gamma]}{\E_{\nu_{\rho}}[\tau \vert \Gamma]}. \ee

\begin{proof}[Proof of Theorems \ref{LLNforEP} and \ref{LDforEP}]
We first prove \eqref{eq:expconc}.
From Theorem \ref{regtimes} and Proposition \ref{momtau} we obtain LDP's 
for both $S_n$ and $X_{S_n}$ with
 rate functions which are only zero at $\E_{\nu_{\rho}}[\tau \vert \Gamma]$ and $\E_{\nu_\rho}[X_{\tau} \vert \Gamma]$, respectively. 
Since $k_t$ is the inverse of $S_n$, it also
satisfies a LDP with a rate function which is
zero only at $\E_{\nu_{\rho}}[\tau \vert \Gamma]^{-1}$ (see Glynn-Whitt \cite{GlWh94}). 
Fix $\epsilon > 0$. From the LDP's for $X_{S_n}$ and $k_t$, we get exponential decay 
of $\P_{\nu_\rho}\left( |t^{-1}X_{S_{k_t}} - v| \ge \epsilon\right)$, while the same is obtained 
for $\P_{\nu_\rho}\left( |X_t-X_{S_{k_t}}| \ge \epsilon t \right)$ from \eqref{sups} and the LDP for $k_t$.
From this, \eqref{eq:expconc} is readily obtained, and the LLN follows by the Borel-Cantelli lemma.  By \eqref{MStocDom}, $v \ge \alpha_0\wedge \alpha_1 - \beta_0 \vee \beta_1 > 1$. Convergence in $L^p$ follows from \eqref{Lpbound}.
\end{proof}

\begin{proof}[Proof of Theorem \ref{CLTforEP}]
Let $\hat{\sigma}^2$ be the variance of $X_{\tau}$ under
$\P_{\nu_\rho}(\cdot | \Gamma)$ which is finite due to \eqref{eq:momtau} and positive since $X_{\tau}$ is not a.s. constant. For the process $(X_{S_k})_{k \in \N}$, a functional CLT with variance $\hat{\sigma}^2$ holds since, by Theorem \ref{regtimes} and \eqref{eq:momtau}, the assumptions of the Donsker-Prohorov invariance principle are satisfied. With a random time change argument as in Section 17 of \cite{Bi68}, we obtain for $(X_{S_{k_t}})_{t \ge 0}$ a functional CLT with variance $\sigma^2 = \hat\sigma^2 \E_{\nu_{\rho}}[\tau \vert \Gamma]^{-1}$. To extend it to $X$, note that
\be{supconv}
\lim_{n \to \infty}n^{-1/2}\sup_{t \le T} \left|X_{nt} - X_{S_{k_{nt}}}\right|
= 0 \quad \P_{\nu_\rho} \text{-a.s. for any } T > 0.
\ee
This follows from Theorem \ref{regtimes}, \eqref{sups} and the LDP for $k_t$ (mentioned in the previous proof), and implies that the Skorohod distance between diffusive rescalings of $X$ and $Z$ goes to zero almost surely as $n \to \infty$.
\end{proof}

\subsection{Einstein Relation: proof of Theorem \ref{ERforEP}}
\label{subsec:ER}

We first show how the speed $v$ is related to the observed density of particles, and that the latter approaches the  density of the environment as $\lambda \downarrow 0$.
\begin{proposition}\label{prop:density}
The limit
\begin{align}\label{eq:density}
\hat\rho(\lambda) &= \lim_{t\to\infty} \frac{1}{t} \int_0^t\E_{\nu_\rho} \left[ \xi_s(X_s) \right ]\di s
\end{align}
exists and satisfies
\begin{align} \label{eq:speeddensity} 
&v(\lambda) = \left[\alpha_1(\lambda)-\beta_1(\lambda) \right]\hat\rho(\lambda) + \left[\alpha_0(\lambda)-\beta_0(\lambda)\right]\left[1-\hat\rho(\lambda)\right],\\
&\lim_{\gamma \downarrow 0}\hat\rho(\lambda) = \rho.  \label{eq:density-perturbation}
\end{align}
\end{proposition}
\begin{proof}
Since $X$ is Markovian under the quenched measure,
\be{martingale}X_t - \int_0^t (\alpha_1-\beta_1)\xi_s(X_s) + (\alpha_0-\beta_0)(1-\xi_s(X_s))\di s \ee
is a martingale under $P_{X}^{\xi}$ for a.e. $\xi$. Hence by Theorem \ref{LLNforEP} the limit in \eqref{eq:density} exists and satisfies \eqref{eq:speeddensity}. We proceed to prove \eqref{eq:density-perturbation}. Write
\begin{align*}
\int_0^t\E_{\nu_\rho}\left[\xi_s\left(X_s\right)\right] \di s = 
& \int_0^t \P_{\nu_\rho}\left(\gamma_s(X_s)\in A_s,\xi_s(X_s) = 1 \right)\di s \\
+ & \int_0^t \P_{\nu_\rho} \left(\gamma_s(X_s)\notin A_s,\xi_s(X_s)=1 \right)\di s.
\end{align*}
The first term is bounded by
\be{timeonmarked} L_t:= \E_{\nu_\rho} \left[ \int_0^t \ind_{\left\{\gamma_s(X_s)\in A_s\right\}} \di s \right], \ee
the expected time spent by the walker on marked agents up to time $t$. 
For the second term, we use Lemma \ref{FreshBernoulli}:
\begin{align*}
\int_0^t \P_{\nu_\rho} \left(\gamma_s(X_s)\notin A_s, \xi_s(X_s) = 1 \right) \di s = 
& \int_0^t \E_{\nu_\rho}\left[ \ind_{\{\gamma_s(X_s)\notin A_s\}}\E_{\nu_\rho}\left[\xi_s\left(X_s\right) \,\middle|\, \mathcal{G}_s \right] \right] \di s  \\
= & \; \rho \int_0^t \P_{\nu_\rho} \left(\gamma_s(X_s)\notin A_s \right) \di s = \rho \left(t - L_t \right). 
\end{align*}
Hence
\be{boundbytime}
\left| \int_0^t \E_{\nu_\rho} \left[\xi_s(X_s)\right]\di s - \rho t \right| \leq L_t. 
\ee
In order to bound $L_t$, consider the total time that the walker spends on top of a single marked agent $x$. 
If $t$ is the time when this agent is marked, the agent will never cross to the right of $Y^t(\gamma^{-1}_t(x))$. 
On the other hand, after time $t$, $X$ will never be to the left of $M-M_t +\gamma^{-1}_t(x)-1$. 
Hence the time spent on the marked agent $x$ is bounded by the total time during which $Y^t(\gamma^{-1}_t(x))$ is to the right of $M-M_t+\gamma^{-1}_t(x)$. 
Writing $t_x = \inf\{t\geq 0 : x \in A_t\}$, we get
\begin{align}\label{boundtimeonA}
 L_t &\leq \sum_{x\in\Z}\E_{\nu_\rho} \left[ \ind_{\{t_x < t\}}\int_{t_x}^\infty \ind_{\{ Y^{t_x}_s(\gamma_{t_x}^{-1}(x)) > M_s-M_{t_x}+\gamma_{t_x}^{-1}(x) \}} \di s \right] \nonumber \\
 &= \E_{\nu_{\rho}}\left[|A_t|\right] \E_{\nu_{\rho}}\left[\int_{0}^\infty \ind_{\{Y^0_s(0) > M_s \}} \di s \right].
\end{align}
When $\lambda$ is small enough, \eqref{drift2} is satisfied, and the term with the integral in \eqref{boundtimeonA} is uniformly bounded by some constant $C \in (0,\infty)$. On the other hand, the number of marked agents $|A_t|$ is bounded by $\hN_t$, so finally we have
\begin{equation*} 
\left| \int_0^t\E_{\nu_\rho}\left[\xi_s(X_s) \right]\di s - \rho t \right| \leq L_t \leq t C \Big(|\alpha_1(\lambda)-\alpha_0(\lambda)| + |\beta_1(\lambda)-\beta_0(\lambda)|\Big) ,
\end{equation*}
proving \eqref{eq:density-perturbation}.
\end{proof}

\begin{proof}[Proof of Theorem \ref{ERforEP}]
Write
\begin{align*}
    \frac{v(\lambda)-v(0)}{\lambda}
& = \frac{(\alpha_1(\lambda)-\beta_1(\lambda)) - (\alpha_1(0)-\beta_1(0))}{\lambda}\hat\rho(\lambda)  \\
& \quad+ (\alpha_1(0)-\beta_1(0))\frac{\hat\rho(\lambda)-\hat\rho(0)}{\lambda}    \\
& \quad+ \frac{(\alpha_0(\lambda)-\beta_0(\lambda)) - (\alpha_0(0)-\beta_0(0))}{\lambda}(1-\hat\rho(\lambda)) \\
& \quad+ (\alpha_0(0)-\beta_0(0))\frac{(1-\hat\rho(\lambda))-(1-\hat\rho(0))}{\lambda}    \\
& = \frac{(\alpha_1(\lambda)-\beta_1(\lambda)) - (\alpha_1(0)-\beta_1(0))}{\lambda}\hat\rho(\lambda) \\
&\quad+ \frac{(\alpha_0(\lambda)-\beta_0(\lambda)) - (\alpha_0(0)-\beta_0(0))}{\lambda}(1-\hat\rho(\lambda)).
\end{align*}
Now take the limit as $\lambda \downarrow 0$ and use \eqref{eq:density-perturbation} to get
\begin{align*}
v'(0) &= \left(\alpha\frac{F_1}{\rho}+\beta\frac{F_1}{\rho}\right)\rho + \left(\alpha\frac{F_0}{1-\rho}+\beta\frac{F_0}{1-\rho}\right)(1-\rho) \\
&=(\alpha+\beta)(F_1+F_0) = \alpha+\beta.
\end{align*}
\end{proof}

\newpage

\end{document}